%% file: FanoNotes.tex
\documentclass{amsart}
\usepackage{amsmath}
\usepackage{amsfonts}
\usepackage{amssymb}
\usepackage{amsthm}
\usepackage{amsopn}
\usepackage{amscd}
\usepackage{amsxtra}
\usepackage{xypic}
\usepackage[arrow,matrix,curve]{xy}
\usepackage{blkarray}
\usepackage{multirow}
\usepackage{bigdelim}

\newtheorem{theorem}{Theorem}[section]
\newtheorem{lemma}[theorem]{Lemma}

\newtheorem{proposition}[theorem]{Proposition}

\theoremstyle{definition}
\newtheorem{definition}[theorem]{Definition}

\newtheorem{example}[theorem]{Example}

\theoremstyle{remark}

\newcommand{\CC}{\mathbb{C}}
\newcommand{\RR}{\mathbb{R}}
\newcommand{\NN}{\mathbb{N}}

\newcommand{\PP}{\mathbb{P}}

\newcommand{\del}{\delta}
\newcommand{\g}{\gamma}
\newcommand{\diff}{\mathrm{d}}
\newcommand{\delm}{\delta_{-}}

\newcommand{\fb}{\mathbf{f}}
\newcommand{\gb}{\mathbf{g}}
\newcommand{\db}{\mathbf{d}}

\newcommand{\Ib}{\mathbf{I}}
\newcommand{\Jb}{\mathbf{J}}

\newcommand{\codim}{\mathrm{codim}}

\newcommand{\Sp}{\, \mathrm{Span}}

\newcommand{\Sym}{\operatorname{Sym}}
\newcommand{\Gr}{\operatorname{Gr}}

\newcommand{\Symd}{\Sym^d(\CC^{n+1})^*}
\newcommand{\Symdb}{\Sym^\db(\CC^{n+1})^*}

\newcommand{\gL}{\Lambda}

\newcommand{\trans}{\mathrm{t}}

\newcommand{\polyLk}{\Gamma_{\gL_0, k'}}

\newcommand{\rank}{\mathrm{rk}\,}
\newcommand{\cdim}{\dim_{\CC}}
\newcommand{\coker}{\operatorname{coker}}

\newcommand{\calA}{\mathcal{A}}
\newcommand{\calB}{\mathcal{B}}

\newcommand{\calJ}{\mathcal{J}}
\newcommand{\calL}{\mathcal{L}}

\newcommand{\calZ}{\mathcal{Z}}

\author{Paul Larsen}
\address{Humboldt-Universit\"at zu Berlin, Institut f\"ur Mathematik, 10099 Berlin, Germany}
\email{larsen@mathematik.hu-berlin.de}

\title{Notes on Fano varieties of complete intersections}

\numberwithin{equation}{section}

\begin{document}
%\selectlanguage{english}
\begin{abstract}
Fano varieties are subvarieties of the Grassmannian whose points parametrize linear subspaces contained in a given projective variety. These expository notes give an account of results on Fano varieties of complete intersections, with a view toward an application in machine learning. The prerequisites have been kept to a minimum in order to make these results accessible to a broad audience. 
\end{abstract}
\maketitle

\input{introNotes}

\input{dmproof}
\bibliographystyle{amsalpha}
\bibliography{bibliography}

\end{document}

%% file: introNotes.tex
\section{Introduction}
The $k$th Fano variety of a projective variety $X \subseteq \PP^n$ is the subvariety of the Grassmannian $\Gr(k,n)$ parametrizing all $k$-planes contained in $X$.\footnote{In addition to this notion of Fano variety, there is another concept in algebraic geometry relating to the positivity of a variety's anticanonical bundle. These two classes of varieties, unfortunately bearing the same name, are in general unrelated.}
These notes give an expanded account of results on Fano
varieties of complete intersections, as studied in \cite{ MR510081, MR1440180,MR1446187, MR1654757}. They are also intended as a companion to the paper
\cite{MLFano}, which extends these results to answer questions arising in machine learning. It is
our hope to make these results more accessible to
non-specialists in algebraic geometry. We assume only basic knowledge of affine and projective varieties, plus familiarity with a few key examples, such as Grassmannians, though even these prerequisites could be replaced by facility with local calculations of the sort that arise in an introductory differential geometry course.

In addition to the above research papers, we refer the reader to \cite{MR1788561, MR2290010} for background on algebraic geometry.  A very nice treatment of Fano varieties of hypersurfaces can be found in \cite{waldronThesis}. We always work over $\CC$, though all results below carry over to any algebraically closed field. In general, we do not distinguish between algebraic varieties and algebraic sets, i.e. we allow varieties to be reducible.

\section{Background}
\label{sec:basicAG}
In this section we introduce tangent spaces to algebraic varieties, coordinates for Grassmannians, Fano varieties and several theorems on fiber dimension.

\subsection{Tangent Spaces}
There are various notions of tangent spaces in algebraic geometry. Since local calculations play a significant role in these notes, we define the tangent space to $X$ at a point $p$ in terms of an affine patch $U$ containing $p$.
\begin{definition}
\label{def:tanSpace}
Let $U \subseteq \CC^{n}$ be an affine variety defined by equations $f_1, \ldots, f_s \in \CC[x_1, \ldots, x_n]$. Then the \emph{Zariski tangent space} to $U$ at a point $p$ is the kernel of the Jacobian matrix, $J = ( \frac{\partial f_i}{\partial x_j}|_p)_{\substack{1 \leq i \leq s}{1 \leq j \leq n}}$.
\end{definition}
In the case of a projective variety defined by homogeneous polynomials $f_1, \ldots, f_s \in \CC[z_0, \ldots, z_n]$, we calculate the tangent space to $V = V(f_1, \ldots, f_s) \subseteq \PP^n$ at a point $p$ by choosing one of the open affine coordinate patches $U_i$ containing $p$, where $U_i = \{[z_0, \ldots, z_n]: z_i \neq 0\} \cong \CC^n$. For notational simplicity, we assume $i=0$, so that the isomorphism $U_0 \cong \CC^n$ is given by the map $[z_0, \ldots, z_n] \to  (z_1/z_0, \ldots, z_n/z_0) = (x_1, \ldots, x_n)$. We dehomogenize the defining equations, and then calculate the tangent space as in Definition \ref{def:tanSpace}. Note that the resulting space is contained in $\CC^n \cong U_0$. If we instead wanted the tangent space defined in the original ambient projective space (sometimes called the \emph{projective tangent space}), we take the its closure in $\PP^n$. 

We now turn to a few examples and basic properties.
\begin{enumerate}
\item For a hypersurface $V(f) \subseteq \CC^n$, the tangent space at a point $p \in V(f)$ is the perpendicular subspace to the vector $\diff f|_p = (\frac{\partial f}{\partial x_1}|_p, \ldots, \frac{ \partial f}{\partial x_n}|_p)$.
\item Since the tangent space of an intersection of hypersurfaces is the intersection of the respective tangent spaces, a variety $V = V(f_1, \ldots, f_s)$ has tangent space given by the intersection of the perpindicular subspaces, $\cap_{i=1}^s (\diff f_i |_p)^\perp$, which is precisely the kernel of the Jacobian matrix above.
\item Let $f = z_0 z_2 - z_1^2$, and let $V = V(f) \subseteq \PP^2$ be the resulting quadric surface. If $p = [1,0,0]$, we take the coordinate patch $U_0$, where the defining equation of $V$ is $f_0 = x_2 - x_1^2$. Then the tangent space at $(0,0)$ is $(0,1)^\perp = V(x_2) \subseteq \CC^2$. Its closure in $\PP^2$ is defined by the equation $z_2 = 0$, and is therefore the projective line $\{[z_0, z_1, 0]: z_0 \neq 0 \textrm{ or } z_1 \neq 0\} \subseteq \PP^2$. Note that this projective subspace coincides with the perpindicular subspace of the formal differential $\diff f|_p = [\frac{\partial f}{\partial z_0}|_p, \frac{\partial f}{\partial z_1}|_p, \frac{\partial f}{\partial z_2}|_p] = [0,0,1]$ in $\PP^2$. 

More generally, if $f \in \CC[z_0, \ldots, z_n]$ is homogeneous of degree $d$, the projective tangent space can be calculated directly from the formal differential $\diff f$. Without loss of generality, assume $p \in V(f) \cap U_0$. By Euler's formula, $d \, f = \sum_{i=0}^n z_i|_p \frac{\partial f}{\partial z_i}|_p$, and since
 $\frac{\partial f_0}{\partial x_i}|_{(x_1, \ldots, x_n)} = \frac{\partial f}{\partial z_i}|_{[1,x_1, \ldots, x_n]}$ for $i =1, \ldots, n$, we obtain $\frac{\partial f}{\partial z_0}|_p = d\, f(p) - \sum_{i=1}^n \frac{\partial f_0}{\partial x_i} x_i = 0$ for $(x_1, \ldots, x_n) \in T_p V$, as both terms on the right vanish. Hence the closure in $\PP^n$ of the Zariski tangent space (i.e. the projective tangent space) coincides with the kernel of the formal differential $\diff f$.
\item If the Jacobian matrix is full-rank at a point $p$, then variety is called \emph{smooth} at $p$. This definition partially agrees with the usual definition from differential geometry, since in differential geometry the non-degeneracy of the Jacobian only gives a necessary condition for smoothness. For example, the Jacobian matrix of $V(y^3 - x^5) \subseteq \RR^2$ is zero at the origin, even though the entire curve is diffeomorphic to $\RR$.
\end{enumerate}

\subsection{Grassmannians and Fano varieties}
The Grassmannian is the algebraic variety $\Gr(k,n)$ parametrizing $k$-dimensional planes in $\PP^n$ (equivalently, $(k+1)$-dimensional subspaces of $\CC^{n+1}$). For example, $\Gr(0,n)$ parametrizes points in $\PP^n$ (equivalently, lines in $\CC^{n+1}$), i.e. $\Gr(0,n) = \PP^n$, while $\Gr(n-1,n)$ parametrizes hyperplanes in $\PP^n$.

Let $e_0, \ldots, e_n$ be a basis of $\CC^{n+1}$, and consider $[\gL_0] = [\Sp(e_0, \ldots, e_k)] \subseteq \Gr(k,n)$ (we will use brackets when we consider this subspace as a point in the Grassmannian, and not as a subspace of $\CC^{n+1}$). Then a coordinate patch $U_0$ about $[\gL_0]$ is given by the row span of all $(k+1)\times(n+1)$ matrices whose first $(k+1)\times(k+1)$ minor is the identity matrix. Every choice of the remaining $(k+1)\times(n-k)$ entries yields a different point in $\Gr(k,n)$, so $U_0 \cong \CC^{(k+1)(n-k)}$.

\begin{example}
Consider $[\gL_0] = [\Sp(e_0, e_1)] \in \Gr(1,3)$. The coordinate patch $U_0$ centered at $[\gL_0]$ consists of the row-spans of the matrices
\begin{equation*}
\begin{pmatrix}
1 & 0 & x_{0,2} & x_{0,3}\\
0 & 1 & x_{1,2} & x_{1,3}
\end{pmatrix}.
\end{equation*}
Suppose we are interested in all lines passing through the points $[1,0,0,0]$. In $U_0$, the coordinates $(x_{0,2}, x_{0,3}, x_{1,2}, x_{1,3})$ must satisfy
\begin{equation*}
\rank
\begin{pmatrix}
1 & 0 & x_{0,2} & x_{0,3}\\
0 & 1 & x_{1,2} & x_{1,3} \\
1 & 0 & 0 & 0
\end{pmatrix} = 2,
\end{equation*}
or equivalently $x_{0,2} = x_{0,3} = 0$. This subvariety is an example of a Schubert cell, namely $\Sigma_{(2,0)} \subseteq \Gr(1,3)$. 
\end{example}

Another class of subvarieties of Grassmannians comes from all $k$-planes contained in a given projective variety $V \subseteq \PP^n$.
\begin{definition}
Let $V \subseteq \PP^n$ be a variety. Then $F_k(V) = \{[\gL] \in \Gr(k,n): \gL \subseteq V\}$ is the $k$th \emph{Fano variety} of $V$.
\end{definition}
\begin{example}
Let $f = z_0 z_2 - z_1 z_3$, and $V = V(f) \subseteq \PP^3$. Then $F_1(V)$ consists of lines on the quadric surface $V$, and is a one-dimensional with two connected components. Note that $f([\lambda_0, \lambda_1, 0, 0]) = 0$ for all $\lambda_0, \lambda_1$, hence $[\gL_0] = [\Sp(e_0, e_1)] \in F_1(V)$. Equations for $F_k(V)$ in the coordinate patch $U_0 \subseteq \Gr(1,3)$ are given by the condition $f([\lambda_0, \lambda_1, \lambda_0 x_{0,2} + \lambda_1 x_{1,2}, \lambda_0 x_{0,3} + \lambda_1 x_{1,3}]) = 0$ for all $\lambda_0, \lambda_1$, i.e. \begin{equation*}
\lambda_0^2 x_{0,2}^2 + \lambda_0 \lambda_1   (x_{1,2} - x_{0,3})+ \lambda_1^2 x_{1,3}= 0
\end{equation*}
for all $\lambda_0, \lambda_1$. Hence $F_1(V)$ is defined locally in $U_0$ by the equations $x_{0,2} = x_{1,3} = x_{1,2} - x_{0,3} = 0$.
\end{example}

\subsection{Incidence correspondences and fiber dimensions}
The Fano variety of a degree $d$ hypersurface in $\PP^n$ can be studied via the following \emph{incidence correspondence}:
\begin{equation}
\label{eq:incCor}
\xymatrix @C=1em{
&I =\ar@{->}[ld]_{p} \ar@{->}[rd]^{q} \{([f], [\gL]): \gL \subseteq V(f)\} &\hspace{-14pt}\subseteq  \PP \Sym^d(\CC^{n+1}) \times \Gr(k,n)\\
\PP \Symd & & \Gr(k,n)},
\end{equation}
where $\Symd$ is the vector space with basis consisting of degree $d$ monomials in the variables $z_0, \ldots, z_n$. Its dimension is $\binom{d+n}{n}$. For example, $\Sym^3(\CC^2)$ has the basis $\{z_0^3, z_0^2 z_1, z_0 z_1^2, z_1^3\}$. Since scaling the defining equation of a hypersurface leaves the hypersurface unchanged, we consider the projectivized vector space of polynomials instead. 

For a given polynomial class $[f] \in \PP \Symd$, the fiber $p^{-1}([f])$ is precisely the Fano variety $F_k(V(f))$. The projection maps $p$ and $q$ are flat, and $q$ is surjective and regular (see \cite{MR1788561}, Sections 4.3 and 6.3). Whether or not $p$ is \emph{dominant} will be an important question for our discussion of Fano varieties. 
\begin{definition}
Let $\phi: V \to W$ be a morphism of varieties. Then $\phi$ is \emph{dominant} if the image $\phi(V)$ is dense in $W$.
\end{definition}

A very useful result is the theorem of the dimension of the fiber:
\begin{theorem}
\label{thm:fibDim1}
Let $\phi: X \to Y$ be a dominant morphism of varieties. Then there exists a non-empty open subset $U \subseteq Y$ such that $U \subseteq \phi(X)$ and for all $y \in U$, $\dim \phi^{-1}(y) = \dim X - \dim Y$. 

Moreover, the function $\mu: p \mapsto \dim (\phi^{-1}(\phi(p)))$ is upper semi-continuous; that is, for all $m \in \NN$, the set $\{p \in X: \mu(p) \geq m\}$ is closed in $X$.
\end{theorem}
For a precise definition of dimension, we refer to \cite{MR1788561} and \cite{MR1322960}. This result can be seen as an algebro-geometric version of the rank-nullity theorem from linear algebra. 

In case the fiber $f^{-1}(y)$ is reducible, then the dimension statement refers to each irreducible component of $f^{-1}(y)$. Another result in this direction implies irreducibility of the domain variety given irreducibility of the target, plus conditions on the fibers.
\begin{theorem}[See \cite{MR1322960}, Exercise 14.3]
\label{thm:fibDim2}
Let $X$ be an irreducible variety, and let $Z \subseteq \PP^n \times X$ be a subvariety (not \emph{a priori} irreducible). Suppose that the fibers of the projection map $\pi_2: Z \to X$ are irreducible of constant dimension. Then $Z$ is irreducible. 
\end{theorem}

We apply these results to the morphisms $q$ and $p$ (at least when $p$ is known to be dominant). The fiber of $q$ over a point $[\gL] \in \Gr(k,n)$ consists of (classes of) homogeneous forms $[f]$ vanishing identically along $\gL$, i.e. for which $\gL \subseteq V(f)$, and therefore is connected and of constant dimension as $\gL$ varies.

\begin{proposition}
\label{prop:incVar}
The incidence variety $I$ is irreducible and smooth of dimension $(k+1)(n-k) + \binom{n+d}{n} - \binom{k+d}{k}-1$. If $p$ is dominant, then 
\begin{enumerate}
\item for all $[f] \in \PP \Sym^d (\CC^{n+1})^*$ the Fano variety $F_k(V(f))$ has at least dimension $(k+1)(n-k) - \binom{k+d}{k}$;
\item if $[f] \in \PP \Sym^d (\CC^{n+1})^*$ is generic, the Fano variety $F_k(V(f))$ has exactly this dimension.
\end{enumerate}
\end{proposition}
\begin{proof}
The irreducibility and dimension claims for $I$ follow from Theorems \ref{thm:fibDim1} and \ref{thm:fibDim2}, as do the dimension statements about $F_k(V(f))$. To prove that $I$ is smooth, note that $I$ is a projective bundle over $\Gr(k,n)$ via $q$.
\end{proof}

For Fano varieties of intersections of hypersurfaces, the main difference is notational. Let $d_1, \ldots, d_s$ be integers greater than or equal to 2, and set $\db = (d_1, \ldots, d_s)$. For any $m \in \NN$, define
\begin{align}
\label{eq:multiNotn}
\begin{split}
\db + m &= (d_1 + m, \ldots, d_s + m), \text{ and}\\
\binom{\db}{m} & = \sum_{i=1}^s \binom{d_i}{m}.
\end{split}
\end{align}
Further set $\Sym^\db(\CC^{n+1})^* = \oplus_{i=1}^s \Sym^{d_I}(\CC^{n+1})^*$. We will write elements of the vector space $\Sym^\db (\CC^{n+1})^*$ as $\fb = (f_1, \ldots, f_s)$, and elements of $\PP \Sym^\db(\CC^{n+1})^*$ as $[\fb]$. We denote the corresponding intersection of hypersurfaces in $\PP^n$ by $V(f)$. Note that while $\PP \Symd$ is a parameter space for hypersurfaces of degree $d$ in $\PP^n$, this is not the case for $\PP \Sym^\db(\CC^{n+1})^*$ and intersections of hypersurfaces in $\PP^n$. For example, if $\fb = (z_1 - z_2, z_0 z_2 - z_1^2)$ and $\gb = ( z_1 - z_2, z_0 z_2(z_0 - z_2))$, then $[\fb] \neq [\gb]$, even though $V(\fb) = V(\gb)  = [1,0,0] \cup [1,1,1]$. More generally, replacing any generator $f_i$ by an element of the ideal generated by the other generators give other examples of this phenomenon.

If $V(\fb)= V(f_1, \ldots, f_s) \subseteq \PP^n$ is the common vanishing locus of homogeneous polynomials of degree $d_1, \ldots, d_s$ respectively, then in the incidence correspondence (\ref{eq:incCor}) replace $\PP \Symd$ with $\PP \Symdb$, and $[f]$ with $[\fb]$. Then the dimension of $\PP \Symdb$ is $\binom{\db + n}{n} - 1$, and the analogue of Proposition (\ref{prop:incVar}) holds with the obvious modifications for the dimensions involved. In the next section, we calculate the dimension of Fano varieties of generic complete intersections by first establishing the dominance of $p$ subject to inequalities in $n$, $k$, and $\db$.

%% file: dmproof.tex
\section{Fano varieties of complete intersections}
\label{sec:FanoCI}

In this section we give an expanded proof of a main result on Fano varieties of complete intersections from
\cite{MR1654757} that is then generalized in \cite{MLFano} to compute
the identifiability index for Stationary Subspace Analysis; see \cite{franzML}. Define for $n, k,s \in \NN$ and $\db = (d_1, \ldots, d_s) \in \NN^s$
\begin{align*}
&\del(n,\db,k)  = (k+1)(n-k) - \binom{\db+k}{k}, \\
&\delm(n,\db,k) = \min\{ \del(n,\db,k), n - 2k - s \}.
\end{align*}

\begin{theorem}[Theorem 2.1, \cite{MR1654757}]
\label{thm:DM}
Let $[\fb] = [(f_1, \ldots, f_s)] \in \PP \Symdb$,
and let $V(\fb)\subseteq \PP^n$ be
the common vanishing locus of the $f_i$.
\begin{enumerate}
\item If $\delm(n,k,s) < 0$, then for generic $[\fb]$ the variety $F_k(V(\fb))$ is empty.
\item If $\delm(n,k,s) = 0$, then for generic $[\fb]$ the variety
  $F_k(V(\fb))$ is non-empty, and smooth of dimension $\del(n, \db, k)$.
\item If $\delm(n,k,s) > 0$, then for generic $[\fb]$ the variety
  $F_k(V(\fb))$ is also connected.
\end{enumerate}
\end{theorem}

In outline, the proof of \ref{thm:DM} is
as follows: to obtain the dimension claims, we apply the theory of the dimension of the fiber
to the map $p$, which will imply that, for a generic $[\fb] \in \PP \,
\Symdb$, the Fano variety $F_k(V_\fb)$ has codimension
$\binom{\db + k}{k}$ in $\Gr(k,n)$, thus giving the dimension
statements of the theorem. To characterize when $p$ is dominant, we observe that if there is a single point $([\fb], [\gL]) \in I$
where $\diff p$ is surjective, then $p$ is in fact
dominant. This fact can be viewed as the algebro-geometric version of the
inverse function theorem. The plan of \cite{MR1654757} is 
to show, under assumptions on $n,\db$ and $k$, that the locus of points where $\diff  p$ is not surjective
forms a codimension $\geq 1$ subvariety of $I_r$, and therefore is not empty.

For the smoothness claims of the theorem, we show that the closure of the image
under $p$ of points $([\fb], [\gL]) \in I$ where $p$ is not
smooth is not all of $\PP \, \Symdb$. Then smoothness of $p$ over a
dense subset (proved in the establishing the dimension results), will imply that the generic fiber is also
smooth. The connectedness of $F_k(V(\fb))$ follows then by applying
the Stein factorization to $p$.

For a fixed $\gL \in Gr(k,n)$, we can
choose coordinates so that
\begin{equation}
\label{eq:canonL}
\gL = \Sp(e_0, \ldots, e_k).
\end{equation}
In the case $\db = (2, \ldots, 2) \in \NN^s$, for example, the fiber $q^{-1}([\gL])$
consists of (the projectivization) of all $s$-tuples of symmetric $(n+1)\times (n+1)$ matrices with
the first $(k+1)\times (k+1)$ submatrix of each matrix identically
zero. In general, the fiber is a projective subspace of has codimension $\binom{\db + k}{k}$ in
$\PP \, \Sym^{\db} (\CC^{n+1})^*$, hence $I$ is a smooth, projective variety
of the same codimension in $\PP \,
\Symdb \times \Gr(k,n)$. 

We first calculate the tangent space of $I$ at a point $([\fb],
[\gL])$. Coordinates for $\Gr(k,n)$ have been discussed in the
previous section; the corresponding coordinates for the tangent space
will be denoted by $(X_{a,b})$, where $0 \leq a \leq k$ and $k+1 \leq b \leq n$.

For $\PP \Symdb$, we first fix some
notation.
If $d \in \NN$, we index the exponents of monomials in $\Symd$ by the
multi-indices $\calJ_{d} = \{ J \in \NN^{n+1}: |J| =
d\}$.
 Choose then the basis for
$\Symd$ given by $\calB_d = \{ z^{J}: J \in \calJ_d\}$. For a multidegree $\db = (d_1,
\ldots, d_s) \in \NN^s$, let $\iota: \Sym^{d_i}(\CC^{n+1})^* \hookrightarrow
\Symdb$ be the canonical inclusion map, and take as a basis for $\Symdb$
\begin{equation}
\label{eq:basis}
\calB_\db = \cup_{i=1}^s \iota(\calB_{d_i}).
\end{equation}
A general element of $\Symdb$ will be denoted as $\sum c_\Jb \, z^\Jb = (\sum
c_{J_1} \, z^{J_1}) \oplus \ldots \oplus (\sum c_{J_s} \,
z^{J_s})$. For the tangent space to $\PP \Symdb$ we therefore can take
as coordinates $( (C_{J_1}), \ldots,
(C_{J_s}) )$, where the entries in the $i$th component run over
elements $J_i \in \calJ_{d_i}$, $1 \leq i \leq s$.

A point $([\fb], [\gL])$ is in $I$ if and only if
\begin{align*}
&f_1(\lambda_0, \ldots, \lambda_k, \,\sum_{a=0}^k \lambda_a x_{a, k+1},
\ldots,\sum_{a=0}^{k} \lambda_a x_{a,n}) 
\equiv 0,\\
&f_2(\lambda_0, \ldots, \lambda_k, \,\sum_{a=0}^k \lambda_a x_{a, k+1},
\ldots,\sum_{a=0}^{k} \lambda_a x_{a,n}) 
\equiv 0,\\ 
& \quad \vdots \\
&f_s(\lambda_0, \ldots, \lambda_k, \,\sum_{a=0}^k \lambda_a x_{a, k+1},
\ldots,\sum_{a=0}^{k} \lambda_a x_{a,n}) 
\equiv 0
\end{align*}
identically for all $\lambda_0, \ldots, \lambda_k \in
\CC$. Differentiating and evaluating at $\gL$ (i.e. substituting
$x_{0,k+1} = \ldots, = x_{k, n} = 0$) yields for each $i \in \{1, \ldots, s\}$ the equation
\begin{gather}
\label{eq:tanI}
\sum_{|J_i| = d_i} C_{J_i} \lambda^{J_i} + \sum_{j=k+1}^{n} \frac{\partial
f_i}{\partial z_j}(\lambda_0, \ldots, \lambda_k, 0, \ldots,
0) \left(\sum_{a=0}^k \lambda_a X_{a,j} \right) =  0.
\end{gather}
The dominance of $p$ could be established by showing that the
projection of solutions to
equations (\ref{eq:tanI}) to the tangent space of $\PP \Symdb$ at $[\fb]$ is full-rank for
generic $([f], [\gL]) \in I$, or merely by finding a single point of $I$ where this projection
is full-rank. The direct method thus boils down to showing certain
matrices with structure are always full-rank. While this approach
lends itself to computer testing, it is less amenable to proof.

Instead note that equations (\ref{eq:tanI}) imply that membership of $( (C_{J_1}), \ldots, (C_{J_s}))$ in the image of
$\diff p$ depends entirely on the existence of linear forms $\ell_{k+1}, \ldots, \ell_n$
on $\gL$ and how they multiply with the partials of the $f_i$
restricted to $\gL$. More precisely, labeling homogeous polynomials of
multidegree $\db$ on
$\gL$ by $\Gamma_\gL(\db)$, we define a multiplication map
\begin{equation}
\label{eq:alpha}
\begin{split}
\alpha_{\fb}: \Gamma_\gL(1)^{n-k} &\to \Gamma_\gL(\db), \\
(\ell_{k+1}, \ldots, \ell_{n}) & \mapsto (\sum_{j=k+1}^{n} \ell_j \left(\frac{\partial
f_1}{\partial z_j}\right)|_\gL, \ldots, \sum_{j=r+1}^{n} \ell_j \left(\frac{\partial
f_s}{\partial z_j}\right)|_\gL ).
\end{split}
\end{equation}

Define next
\begin{align*}
Z_k &=\overline{\{([\fb], [\gL]): p \textrm{ not smooth at }([\fb], [\gL]) \}}
\subseteq I \\
&= \overline{\{ ([\fb], [\gL]): (\diff p)|_{([\fb], [\gL])  } \textrm{ not
  full-rank}\}} \subseteq I,
\end{align*}
where the bar indicates algebraic closure. Our first goal then is to
show that for $n, k,$ and $\db$ as in the
statement of Theorem \ref{thm:DM}, $\codim(Z_k, I) > 0$. By virtue
of the previous discussion, this question can be reformulated via
polynomials on $\gL$, as formulated in the following.
\begin{lemma}
\label{lemma:DM22}
 A point $([\fb], [\gL]) \in I$ is in $Z_k$ if and only
  if $\alpha_{\fb}:\Gamma_\gL(1)^{n-k} \to \Gamma_\gL(\db)$ is not surjective.
\end{lemma}

The map $\alpha_{\fb}$ can be divided up as a composition
of simpler multiplication maps. Define
\begin{align*}
\mu: \Gamma_\gL(1) \times \Gamma_\gL(\db - 1) &\to \Gamma_\gL(\db) \\
(\ell, (g_1, \ldots, g_s)) & \mapsto (\ell g_1, \ldots, \ell g_s),
\end{align*}
so that the image of $\alpha_{\fb}$ is just the sum over $j=1, \ldots,
s$ of images of the map $\mu$ taking, for fixed $j$, $g_i = \frac{\partial f_i}{\partial z_j}|_\gL$.
Now the non-surjectivity of $\alpha_{\fb}$ can be characterized by the
existence of a hyperplane in $\Gamma_\gL(\db)$ (i.e. an element of
$\Gamma_\gL(\db)^*$) such that as $\ell$ ranges over all linear forms
on $\gL$, the image under $\mu$ of each vector of
partials $(\frac{\partial f_1}{\partial z_j}|_\gL, \ldots,
\frac{\partial f_s}{\partial z_j}|_\gL)$ sits in this hyperplane for
all $i=1, \ldots, s$. For $h \in \Gamma_\gL(\db)^*$, define
\begin{equation*}
{\calA}_{[h]} = \{(g_1, \ldots, g_s) \in \Gamma_\gL(\db - 1): \mu(\ell, (g_1,
\ldots, g_s)) \in h^\perp \textrm{ for all }\ell \in \Gamma_\gL(1)\},
\end{equation*}
and $\calZ \subseteq q^{-1}([\gL]) \times \PP \Gamma_\gL(\db)^*$ as
\begin{equation*}
\calZ = \{([\fb], [h]): \left(\frac{\partial
  f^{(j)}}{\partial z_i}|_\gL\right)_{1 \leq j \leq s} \in \calA_{[h]}
\textrm{ for all } k+1 \leq i \leq n\}.
\end{equation*}
\noindent Hence as $h^\perp$ ranges over all hyperplanes in $\Gamma_\gL(\db)$,
the $[\fb]$ such that the partials of $\fb$ restricted to $\gL$ lie in
$\calA_{[h]}$ 
sweep out points of $q^{-1}([\gL])\cap Z_k$, i.e. 
$pr_1(\calZ) = q^{-1}([\gL]) \cap Z_k$,
and therefore
\begin{equation}
\label{eq:codimIneq1}
\codim (q^{-1}([\gL]) \cap Z_k, q^{-1}([\gL]) \cap I) \geq \codim
(\calZ,  q^{-1}([\gL]) \times \PP \Gamma_\gL(\db)^*).
\end{equation}

We next obtain a lower bound for the term on the right. Recall that $\gL$ is defined in coordinates as in Equation
(\ref{eq:canonL}), so that we may identify elements of $\Gamma_\gL(d)$ with
homogeneous degree $d$ polynomials in $z_0, \ldots, z_k$. For $h
\neq 0$,  it follows that $\codim(\calA_{[h]}, \Gamma_\gL(\db - 1))$ is at most $k+1$, since
$z_0, \ldots, z_k$ form a basis for $\Gamma_\gL(1)$, and $\calA_{[h]}$ is then
defined by the $k+1$ equations $h ( \mu(z_i, (g_1, \ldots, g_s)) =
0$. For $t=1, \ldots, k+1$, we then set
\begin{equation*}
{\calL}(t) = \overline{\{h \in \Gamma_\gL(\db)^*: \codim(\calA_{[h]}, \Gamma_\gL(\db - 1)) = t\}},
\end{equation*}
so that
\begin{equation}
\label{eq:decompZ}
\calZ = \bigcup_{t=1}^{k+1} \bigcup_{[h] \in \PP \calL(t)} \{([\fb], [h]): \left(\frac{\partial
  f_j}{\partial z_i}|_\gL\right)_{1 \leq j \leq s} \in \calA_{[h]}
\textrm{ for all } k+1 \leq i \leq n\}.
\end{equation}
To estimate the codimension of each term on right-hand side, fix $t \in \{1,
\ldots, k+1\}$, and $[h]
\in \PP \calL(t)$. By definition,
the membership  $\left(\frac{\partial
  f^{(j)}}{\partial z_i}|_\gL\right)_{1 \leq j \leq s} \in
\calA_{[h]}$ is a codimension $t$ condition in $\Gamma_\gL(\db - 1)$
for each $i \in \{k+1, \ldots, n\}$. The defining conditions are
linear equations on $\Gamma_\gL(\db -
1)$, each of which factors through a projection of $pr_i: \Gamma_\gL(\db -
1) \to \Gamma_\gL(d_i - 1)$. Therefore the defining equations for
different $i, i' \in \{k+1, \ldots, n\}$ are pairwise independent, and
the codimensions from each vector $\left(\frac{\partial
  f_j}{\partial z_i}|_\gL\right)_{1 \leq j \leq s}$ add to total $t(n-k)$. It follows that
\begin{equation}
\label{eq:DM24}
\codim
(\calZ,  q^{-1}([\gL]) \times \PP \Gamma_\gL(\db)^*) \geq \min_{1 \leq
  t \leq k+1} \{t(n-k) - \dim \PP \calL(t)\}.
\end{equation}

%%%%%%%%%%%%%%%%%%%%%%%%%%%%%%%%%%%%%%%%%%%%%%%%%%%%%%%
%%%%%%%%%%%%%%%%%%%%%%%%%%%%%%%%%%%%%%%%%%%%%%%%%%%%%%%

\begin{lemma}
\label{lemma:DM28}
For $t \in \{1, \ldots, k+1\}$, 
\begin{equation*}
\dim \PP \calL(t) \leq t(k - t + 1) + \binom{\db + t - 1}{t-1}  -1.
\end{equation*}
\end{lemma}
\begin{proof}
Note that for $t=k+1$, the above inequality is trivially fulfilled, so we assume $t \in \{1, \ldots, k\}$.
Represent the binear map $\mu: \Gamma_\gL(1) \times \Gamma_\gL(\db
- 1) \to \Gamma_\gL(\db)$ via basis choices for $\Gamma_\gL(1)$, $\Gamma_\gL(\db
- 1)$, and $\Gamma_\gL(\db)$ described in Equation (\ref{eq:basis}),
except now for polynomials in $z_0, \ldots, z_k$. We will order the
basis elements reverse-lexicographically.

Let $M(\mu)$ be the matrix representation of $\mu$ with respect to the
bases $\calB_1$, $\calB_{(\db - 1)}$, and $\calB_{\db}$. Then $M(\mu)$
is a $k+1$ by $\binom{(\db-1) +k}{k}$ matrix with entries in
$\Gamma_\gL(\db)$.
\begin{example}
\label{ex:M}
Let $n=4$, $k=1$, and $\db = (2,2)$, so that $M(\mu)$ is a $2 \times
4$ matrix with entries in $\Gamma_\gL(2,2)$.

\begin{equation*}
\begin{blockarray}{ccc|cc}
& z_0 \oplus 0 & z_1 \oplus 0 & 0 \oplus z_0 & 0 \oplus z_1\\
\begin{block}{c(cc|cc)}
z_0 & z_0^2 \oplus 0 &	 z_0 z_1 \oplus 0& 0 \oplus z_0^2 & 0 \oplus
z_0 z_1	\\
z_1 & z_0 z_1 \oplus 0 & z_1^2 \oplus 0 & 0 \oplus z_0 z_1 & 0 \oplus z_1^2\\
\end{block}
\end{blockarray}.
\end{equation*}

If $h = (z_0^2)^* \oplus (z_0^2)^* \in \Gamma_\gL(2,2)^*$, then $h(M(\mu))$ is
\begin{equation*}
\begin{blockarray}{ccc|cc}
& z_0 \oplus 0 & z_1 \oplus 0 & 0 \oplus z_0 & 0 \oplus z_1\\
\begin{block}{c(cc|cc)}
z_0 & 1 &	 0 & 1 & 0 	\\
z_1 & 0 & 0 & 0 & 0\\
\end{block}
\end{blockarray}.
\end{equation*}

If instead $h = (z_0 z_1)^* \oplus 0 \in \Gamma_\gL(2,2)^*$, then $h(M(\mu))$ is
\begin{equation*}
\begin{blockarray}{ccc|cc}
& z_0 \oplus 0 & z_1 \oplus 0 & 0 \oplus z_0 & 0 \oplus z_1\\
\begin{block}{c(cc|cc)}
z_0 & 0 & 1 & 0 & 0 	\\
z_1 & 1 & 0 & 0 & 0\\
\end{block}
\end{blockarray}.
\end{equation*}

If $k=3$ and $\db = (3,3)$ then the transpose of the matrix $M(\mu)$ is
\begin{displaymath}
M(\mu)^\trans =
  \begin{blockarray}{cc cccc cccc}
    & z_0 & z_1 & z_2 & z_3\\
    \begin{block}{c(c cccc cccc@{\hspace*{5pt}})}
    z_0^2 \oplus 0 & z_0^3 \oplus 0 & z_0^2 z_1 \oplus 0 & z_0^2 z_2 &
    z_0^2 z_3 \oplus 0 \\
    z_0 z_1 \oplus 0&z_0^2 z_1 \oplus 0 & z_0 z_1^2 \oplus 0 & z_0 z_1
    z_2 \oplus 0& z_0 z_1 z_3 \oplus 0\\
    z_1^2 \oplus 0 & z_0 z_1^2 \oplus 0 & z_1^3 \oplus 0 & z_1^2 z_2
    \oplus 0& z_1^2 z_3 \oplus 0\\
    z_0 z_2 \oplus 0 & z_0^2 z_2 \oplus 0& z_0 z_1 z_2 \oplus 0 & z_0
    z_2^2 \oplus 0& z_0 z_2 z_3  \oplus 0\\
z_1 z_2 \oplus 0 & z_0 z_1 z_2 \oplus 0& z_1^2 z_2 \oplus 0 & z_1 z_2^2
\oplus 0& z_1 z_2 z_3 \oplus 0\\
z_2^2 \oplus 0 & z_0 z_2^2 \oplus 0 & z_1 z_2^2 \oplus  & z_2^3 \oplus
0& z_2^2 z_3 \oplus 0\\
z_0 z_3 \oplus 0 & z_0^2 z_3 \oplus 0 & z_0 z_1 z_3 \oplus 0 & z_0 z_2
z_3 \oplus 0& z_0 z_3^2 \oplus 0\\
z_1 z_3 \oplus 0 & z_0 z_1 z_3 \oplus 0 & z_1^2 z_3 \oplus 0 & z_1 z_2
z_3 \oplus 0& z_1 z_3^2 \oplus 0\\
z_2 z_3 \oplus 0 & z_0 z_2 z_3 \oplus 0& z_1 z_2 z_3 \oplus 0 & z_2^2
z_3 \oplus 0& z_2 z_3^2 \oplus 0\\
z_3^2 \oplus 0 & z_0 z_3^2 \oplus 0& z_1 z_3^2 \oplus 0 & z_2 z_3^2
\oplus 0& z_3^3 \oplus 0\\
 %   \cline{1-10}% don't use \hline
  0 \oplus  z_0^2   & 0 \oplus z_0^3  & 0 \oplus z_0^2 z_1  & 0 \oplus
  z_0^2 z_2 & 0 \oplus
    z_0^2 z_3  \\
  0 \oplus  z_0 z_1& 0 \oplus z_0^2 z_1  & 0 \oplus  z_0 z_1^2  & 0 \oplus  z_0 z_1
    z_2 & 0 \oplus  z_0 z_1 z_3 \\
0 \oplus    z_1^2  & 0 \oplus  z_0 z_1^2  & 0 \oplus  z_1^3  & 0 \oplus  z_1^2 z_2
    & 0 \oplus  z_1^2 z_3 \\
 0 \oplus   z_0 z_2  & 0 \oplus  z_0^2 z_2 & 0 \oplus  z_0 z_1 z_2  & 0 \oplus  z_0
    z_2^2 & 0 \oplus  z_0 z_2 z_3  \\
0 \oplus z_1 z_2  & 0 \oplus  z_0 z_1 z_2 & 0 \oplus  z_1^2 z_2  & 0 \oplus  z_1 z_2^2
& 0 \oplus  z_1 z_2 z_3 \\
0 \oplus z_2^2 & 0 \oplus  z_0 z_2^2  & 0 \oplus  z_1 z_2^2 & 0 \oplus
z_2^3 & 0 \oplus  z_2^2 z_3 \\
0 \oplus z_0 z_3 & 0 \oplus  z_0^2 z_3  & 0 \oplus  z_0 z_1 z_3  & 0 \oplus  z_0 z_2
z_3 & 0 \oplus  z_0 z_3^2 \\
0 \oplus z_1 z_3 & 0 \oplus  z_0 z_1 z_3  & 0 \oplus  z_1^2 z_3  & 0 \oplus  z_1 z_2
z_3 & 0 \oplus  z_1 z_3^2 \\
0 \oplus z_2 z_3   & 0 \oplus  z_0 z_2 z_3 & 0 \oplus  z_1 z_2 z_3  & 0 \oplus  z_2^2
z_3 & 0 \oplus  z_2 z_3^2 \\
0 \oplus z_3^2   & 0 \oplus z_0 z_3^2& 0 \oplus z_1 z_3^2 & 0 \oplus
z_2 z_3^2 & 0 \oplus z_3^3\\
    \end{block}
  \end{blockarray}.
\end{displaymath}
\hfill $\Diamond$
\end{example}

In general, an element $g =\sum c_\Jb \, z^\Jb = (\sum
c_{J_1} \, z^{J_1}) \oplus \ldots \oplus (\sum c_{J_s} \,
z^{J_s}) \in \Gamma_\gL(\db-1)$ is in
  $\calA_{[h]}$ if and only if, for all $i \in \{0, \ldots, k\}$,
\begin{align*}
0& = h\left(\sum c_\Jb \,\mu(z_i, z^\Jb)\right)\\
& = h  \left(\sum
c_{J_1} \, z^{J_1+e_i} \oplus \ldots \oplus \sum c_{J_s} \,
z^{J_s + e_i}\right) = 0
\end{align*}
i.e. if and only
  if the $g$ is in the kernel of
  $h(M(\mu))$. Hence $\codim(\calA_{[h]}, \Gamma_\gL(\db-1))$ equals
  the rank of $h(M(\mu))$.

Consider now $h = \sum \g_{I_1} (z^{I_1})^* \oplus \ldots \oplus \sum
\g_{I_s} (z^{I_s})^* \in \polyLk(\db)^*$. Then the matrix entry in row $z_i$
(where $0
\leq i \leq k$) and column $z^\Jb$  (where $z^\Jb \in \iota(\calB_{d_j-1})$,
so $\Jb =  \ldots \oplus 0
\oplus J_j \oplus 0 \oplus \ldots $
for some $1 \leq j \leq s$) is
\begin{equation*}
h(M(\mu_1))_{i, \Jb} = \g_{ \ldots \oplus 0 \oplus J_j + e_i \oplus 0 \oplus \ldots}.
\end{equation*}
By the previous
paragraph, multiplication by $h(M(\mu_1))$ from the left induces a
morphism 
\begin{align*}
\phi: \calL(t) \setminus \calL(t-1) &\to \Gr(k-t, k) \\
 h &\mapsto \PP \coker (h(M(\mu)))
\end{align*}
To finish the proof of the lemma, we
obtain an upper bound on the dimension of the fibers of
$\phi$. Let $[\Sp(v_0, \ldots, v_{k-t})] \in \Gr(k-t,k)$. For each $0 \leq l \leq k-t$, we can write $v_l = (v_{0 \, l}, \ldots, v_{k
  \, l})$, with $v_{j \, l} = 0$ if $j < l$ and $v_{l \, l} \neq
0$.

\begin{example}
\label{ex:M2}
Continuing from the last case given in Example \ref{ex:M}, with $k=3$
and $\db = (3,3)$, for a fixed $h  \in \Gamma_\gL(3)^* \oplus \Gamma_\gL(3)^*$, the first
and last few defining equations of $\phi(h)$ as a subspace of $\PP^3$ are
\begin{align*}
&z_0 \g_{(3,0,0,0) \oplus 0} + z_1 \g_{(2,1, 0,0) \oplus 0} + z_2
\g_{(2, 0, 1, 0), \oplus 0} + z_3 \g_{(2,0,0,1) \oplus 0} = 0, \\
&z_0 \g_{(2,1,0,0) \oplus 0} + z_1 \g_{(1,2,0,0) \oplus 0} + z_2
\g_{(1,1,1,0) \oplus 0} + z_3 \g_{(1,1,0,1) \oplus 0} = 0, \\
&\; \ldots \\
& z_0 \g_{ 0 \oplus (1, 0, 1, 1)  } + z_1 \g_{ 0 \oplus (0, 1, 1, 1)  } + z_2
\g_{ 0 \oplus (0, 0, 2, 1)  } + z_3 \g_{ 0 \oplus (0, 0, 1, 2)  } = 0,
\\
& z_0 \g_{0 \oplus (1, 0, 0, 2)} + z_1 \g_{0 \oplus (0,1,0,2)} + z_2
\g_{ 0 \oplus (0,0,1,2)} + z_3 \g_{ 0 \oplus (0,0,0,3)} = 0.
\end{align*}

Taking now $t=1$, the fiber
$\phi^{-1}(\Sp(v_0, v_1, v_2))$ consists of all $\sum \g_{\Ib}
(z^{\Ib})^* \in \polyLk(\db)$ satisfying

\begin{align*}
&v_{0,0} \g_{(3,0,0,0) \oplus 0} + v_{1,0} \g_{(2,1, 0,0) \oplus 0} + v_{2,0}
\g_{(2, 0, 1, 0), \oplus 0} + v_{3,0} \g_{(2,0,0,1) \oplus 0} = 0, \\
&v_{0,0} \g_{(2,1,0,0) \oplus 0} + v_{1,0} \g_{(1,2,0,0) \oplus 0} + v_{2,0}
\g_{(1,1,1,0) \oplus 0} + v_{3,0} \g_{(1,1,0,1) \oplus 0} = 0, \\
&\; \ldots \\
& v_{0,0} \g_{ 0 \oplus (1,1,0,1)  } + v_{1,0} \g_{ 0 \oplus (0,2,0,1)  } + v_{2,0}
\g_{ 0 \oplus (0,1,1,1)  } + v_{3,0} \g_{ 0 \oplus (0, 1, 0, 2)  } = 0, \\
& v_{0,0} \g_{ 0 \oplus (1, 0, 1, 1)  } + v_{1,0} \g_{ 0 \oplus (0, 1, 1, 1)  } + v_{2,0}
\g_{ 0 \oplus (0, 0, 2, 1)  } + v_{3,0} \g_{ 0 \oplus (0, 0, 1, 2)  }
= 0, \\
&v_{1,1} \g_{(2,1, 0,0) \oplus 0} + v_{2,1}
\g_{(2, 0, 1, 0), \oplus 0} + v_{3,1} \g_{(2,0,0,1) \oplus 0} = 0, \\
& v_{1,1} \g_{(1,2,0,0) \oplus 0} + v_{2,1}
\g_{(1,1,1,0) \oplus 0} + v_{3,1} \g_{(1,1,0,1) \oplus 0} = 0, \\
&\; \ldots \\
&  v_{1,1} \g_{ 0 \oplus (0,2,0,1)  } + v_{2,1}
\g_{ 0 \oplus (0,1,1,1)  } + v_{3,1} \g_{ 0 \oplus (0, 1, 0, 2)  } = 0, \\
&   v_{1,1} \g_{ 0 \oplus (0, 1, 1, 1)  } + v_{2,1}
\g_{ 0 \oplus (0, 0, 2, 1)  } + v_{3,1} \g_{ 0 \oplus (0, 0, 1, 2)  }
= 0,
\\
& v_{2,2} \g_{(2, 0, 1, 0), \oplus 0} + v_{3,2} \g_{(2,0,0,1) \oplus 0} = 0, \\
& v_{2,2}
\g_{(1,1,1,0) \oplus 0} + v_{3,2} \g_{(1,1,0,1) \oplus 0} = 0, \\
&\; \ldots \\
& v_{2,2}
\g_{ 0 \oplus (0,1,1,1)  } + v_{3,2} \g_{ 0 \oplus (0, 1, 0, 2)  } = 0, \\
& v_{2,2}
\g_{ 0 \oplus (0, 0, 2, 1)  } + v_{3,2} \g_{ 0 \oplus (0, 0, 1, 2)  }
= 0,
\end{align*}

\hfill $\Diamond$
\end{example}
To bound the dimension of the solution set of these equations, recall that $\calB_{\db, [e_k]}$ is ordered reverse-lexicographically (with all
elements in the $i$th component coming before
those of the $j$th copy when $i < j$). Following \cite{MR1446187}, we index the equations
for $\phi^{-1}(\Sp(v_0, \ldots, v_{k-t}))$ as follows: to each
equation arising from a generating vector $v_l$, we associate the index
$\Ib$ of the first non-zero coefficient $\g_{\Ib}$ appearing, so
that the equation's label is $(\Ib, l)$. Note that not all labels will
occur: in the equations from Example \ref{ex:M2} the label $( (0,2,0,1)
\oplus 0,
0)$ does not occur.

From the matrix of this system of equations (rows labeled as above by
$(\Ib, l)$, columns by exponents $\Ib$ such that $z^\Ib \in \calB_{\db, [e_k]}$), we define a non-zero
minor $M$ of this matrix.
To select the rows of $M$, define first the following collections of indices for $0 \leq j \leq k-t$:
\begin{align*}
A_{j,1} &= \{\Ib = I_1 \oplus 0  \oplus \ldots \oplus 0: (I_1)_0 = (I_1)_1 = \ldots
(I_1)_{j-1} = 0, (I_1)_j \neq 0\}, \\
& \ldots \\
A_{j,s} & =  \{\Ib = 0 \oplus \ldots \oplus 0  \oplus I_s: (I_s)_0 = (I_s)_1 = \ldots
(I_s)_{j-1} = 0, (I_1)_j \neq 0\},
\end{align*}
and set $A_j = A_{j,1} \cup \ldots \cup A_{j,s}$.
The rows of the minor $M$ are then (in ordering as above)
\begin{equation*}
\{ (\Ib, 0): \Ib \in A_0\} \cup \{ (\Ib, 1): \Ib \in A_1 \} \cup
\ldots \cup  \{(\Ib, k-t): \Ib \in A_{k-t} \},
\end{equation*}
and the columns are
\begin{equation*}
\{\Ib: \Ib \in A_0\} \cup \ldots \cup \{\Ib: \Ib \in A_{k-t}\}.
\end{equation*}

To calculate the size of this minor, suppose first that $\Ib = I_1
\oplus 0 \oplus \ldots \oplus 0 \in A_j$. Then the choice of $I_1$ satisfying $(I_1)_0 = (I_1)_1 = \ldots = (I_1)_{j-1} = 0$
and $(I_1)_j \neq 0$ corresponds to the choice of monomials of degree $d_1
-1$ in variables $z_j, \ldots, z_k$. Moreover, all such
monomials (rather, their exponents) appear in the our labeling for the
system of equations. There are therefore $\binom{k - j +  d_1 -
  1}{d_1-1} = \binom{k-j + d_1 - 1}{k-j}$ such monomials with
exponents of the form
$\Ib = I_1 \oplus 0 \oplus \ldots \oplus 0$ if $j < k$. In general then, there are $\sum_{i=1}^s \binom{k-j+
  d_i - 1}{k-j} = \binom{k-j + (\db - 1)}{k-j}$ monomials in $A_{j}$.

Hence the size of the selected minor $M$ is

\begin{equation}
\label{eq:minor}
\sum_{j=0}^{k-t} |A_j| = \sum_{j=0}^{k-t}  \binom{k-j + (\db -
  1)}{k-j} = \binom{\db + k}{k} - \binom{\db + t -1}{t-1},
\end{equation}
where we have used the identity $\sum_{\ell = 0}^{m}\binom{r-\ell}{m - \ell}
= \binom{r+1}{m}$.

By construction, $M$ is an upper-triangular matrix with diagonal $v_{j
  j} \neq 0$, hence the rank of the matrix of defining equations
for $\phi^{-1}(\Sp(v_0, \ldots, v_{k-t}))$ is at least the number
above, and therefore
\begin{equation*}
\cdim \phi_1^{-1}(\Sp(v_0, \ldots, v_{k-t})) \leq \binom{\db + k}{k}
- s -
\rank(M)
 =  \binom{\db + t -1}{t-1} - s.
\end{equation*}

Using the theorem of the dimension of the fiber, the result follows:
\end{proof}

Inequality \ref{eq:DM24} combined with the above now implies that
\begin{equation*}
\codim (Z_r, I_r) \geq \min_{1 \leq t \leq k+1} \{ t(n - 2k + t - 1) -
\binom{\db + t - 1}{t - 1}\} + 1.
\end{equation*}
Define the expression over which we minimize as
\begin{equation*}
\psi(t) = t(n - 2k + t - 1) -
\binom{\db + t - 1}{t - 1}.
\end{equation*}

%%%%%%%%%%%%%%%%%%%%%%%%%%%%%%%%%%%%%%%%%%%%%%%%%%%%%%%
%%%%%%%%%%%%%%%%%%%%%%%%%%%%%%%%%%%%%%%%%%%%%%%%%%%%%%%
\begin{lemma}
If $\db \neq (2)$, then $\psi(t)$ is concave on $[1, \ldots, \infty)$,
while if $\db = (2)$ and $\delm(n, \db, k) \geq 0$, then $\psi$
is increasing.
\end{lemma}
\begin{proof}
Recall that we assume all $d_i \geq 2$. Consider the first and second forward differences of $\psi$:
\begin{align*}
\Delta_1(\psi)(t) &= \psi(t+1) - \psi(t) = 2t + n-2k - \binom{\db + t
  - 1}{t}, \\
\Delta_2(\psi)(t) & = \Delta_1(\Delta_1(\psi))(t) = 2 - \binom{\db + t
  - 1}{t+1}.
\end{align*}
If $\db \neq (2)$, then either $\db=(d)$ with $d \geq 3$ or $\db =
(d_1, \ldots, d_s)$ with $s \geq 2$. In both
cases $\Delta_2(\psi_2)(t) < 0$ for all $t \geq 1$, proving the
first assertion.

If now $\db = (2)$, then $\Delta_1(\psi)(t) = t + n - 2k - 1$. The
assumption $\delm(n, \db, k) \geq 0$ implies that $n - 2k - 1 \geq 0$, so the
first forward difference is strictly positive for all $t \geq 1$,
proving the second claim.
\end{proof}
%%%%%%%%%%%%%%%%%%%%%%%%%%%%%%%%%%%%%%%%%%%%%%%%%%%%%%%
%%%%%%%%%%%%%%%%%%%%%%%%%%%%%%%%%%%%%%%%%%%%%%%%%%%%%%%

In either case we obtain the crucial codimension estimate
\begin{equation}
\label{eq:codimZrIr}
\codim(Z_r, I_r) \geq \min\{\psi(1), \psi(r+1)\} + 1 = \delm(n, \db, k) - 1.
\end{equation}

We now finish the proof of the dimension statements of Theorem
\ref{thm:DM}. Suppose that $\delm(n, \db, k) < 0$; then if $\db \neq
(2)$, it follows that
$\del(n,\db,k) < 0$, and so the dimension of $I_r$ is strictly less than that
of $\PP \Symdb$. Hence for general $[\fb] \in \PP \Symdb$ the
corresponding fiber of $p$ will be empty. If $\db = (2)$, then $2k \geq
n$. Supposing $\gL$ has the form of Equation \ref{eq:canonL}, the
defining equation
$f$ has the form
\begin{equation*}
f = z_{k+1} \ell_{k+1} + z_n \ell_n,
\end{equation*}
with $\ell_i \in (\CC^{n+1})^*$. Since $n-k \leq k$, the linear
forms $\ell_i$ have a common zero on $\gL$, which implies that
the quadric defined by $f$ is singular at this point, giving a
contradiction for generic
$f$.
%???

If instead $\delm(n, \db, k) \geq 0$, then Equation (\ref{eq:codimZrIr}) implies
the existence of a point in $I$ where $\diff p$ is surjective, and
so $p$ is dominant. By the theorem of the dimension of the fiber, Theorem
\ref{thm:fibDim1}, for generic $[\fb] \in \PP \Symdb$, the fiber
$F_k(V(\fb))$ has the expected dimension $\del(n, \db, k)$.

For the smoothness statements, we first define $\Delta_k$ as the
closure of $p(Z_k)$, and, replacing in the incidence correspondence
$I$ the Grassmannian $\Gr(k,n)$ with $\Gr(k-1,n)$, define
$\Delta_{k-1}$ as the closure of $p(Z_{k-1})$ (by convention we take
$\Delta_{-1} = \emptyset$.

%%%%%%%%%%%%%%%%%%%%%%%%%%%%%%%%%%%%%%%%%%%%%%%%%%%%%%%%
\begin{proposition}
\label{prop:DM27}
There is a strict containment, $\Delta_{k} \subsetneq \PP \, \Sym^{\db} (\CC^{n+1})^*$.
\end{proposition}
% \begin{proof}
% By Proposition \ref{prop:Dr}, $\Delta_k - \Delta_{k-1} \neq \PP\, \Sym^\db
% (\CC^{n+1})^*$, hence by induction on $k$, the result follows.
% \end{proof}
%%%%%%%%%%%%%%%%%%%%%%%%%%%%%%%%%%%%%%%%%%%%%%%%%%%%%%%%
We will prove Proposition \ref{prop:DM27} using the following lemma.
%%%%%%%%%%%%%%%%%%%%%%%%%%%%%%%%%%%%%%%%%%%%%%%%%%%%%%%%
\begin{lemma}
Let $([\fb], [\gL])$ be an element of $Z_k - p^{-1}(\Delta_{k-1})$,
and let $h$ be a non-trivial linear form on $\Gamma_{\gL}(\db)$ such
that the image of $\alpha_{[\fb]}$ is contained in $[h]^\perp$. Then
$\calA_{[h]}$ has codimension $k+1$ in $\Gamma_\gL(\db - 1)$.
\end{lemma}
\begin{proof}
Since $\codim(\calA_{[h]}, \Gamma_\gL(\db - 1))$ equals the rank of
the matrix $h (M(\mu))$ (see the paragraph following Example
\ref{ex:M}), we assume for a contradiction that $h(M(\mu))$ has rank
less than $k+1$. By a linear change of coordinates, we may assume
that the final row of $h(M(\mu))$ is identically zero, i.e. $h( z_k b)
= 0$ for all $b \in \calB_{(\db -1)}$.

Defining $\gL' \subseteq \gL$ by $z_k = 0$, $h$ induces a linear
form $h'$ on $\Gamma_{\gL'}(\db)$. Define next
\begin{equation*}
\alpha'_{[\fb]}: \Gamma_{\gL'}(1)^{n - k + 1} \to \Gamma_{\gL'}(\db)
\end{equation*}
in analogy to $\alpha_{[\fb]}$ of Equation (\ref{eq:alpha}). Then
$\alpha'_{[\fb]} (\{0 \} \times \Gamma_{\gL'}(1)^{n-k}) \subseteq
h^{\perp}$, and since every monomial of $\frac{\partial
  f^{(j)}}{\partial z_k}$ will contain a factor $z_i$, $i \in \{k+1,
  \ldots, n\}$, this partial derivative vanishes
on $\gL'$. It follows that $h'$ annihilates the entire image of
$\alpha'_{[\fb]}$, which implies that $[\fb] \in \Delta_{k-1}$, a contradiction.

\end{proof}

%%%%%%%%%%%%%%%%%%%%%%%%%%%%%%%%%%%%%%%%%%%%%%%%%%%%%%%%
\begin{proof}[Proof of Proposition \ref{prop:DM27}]
It follows immediately that $q^{-1}([\gL]) \cap (Z_k -
p_r^{-1}(\Delta_{k-1})) \subseteq pr_1 (\calZ_{k+1})$. For the
dimension inequality, by the discussion preceding Equation
(\ref{eq:DM24}), it follows that
\begin{equation}
\begin{split}
\dim \overline{(Z_k - pr^{-1}(\Delta_{k-1}))}  &= \dim I   - \codim
(\overline{(Z_r - pr^{-1}(\Delta_{k-1}))}, I) \\
& \leq \dim \PP \, \Sym^\db (\CC^{n+1})^* + \dim \Gr(k, n) - \binom{\db + k}{k} \\
& \quad 
- (k+1)(n-k) + \dim \PP \, \Gamma_\gL(\db) \\
& \leq \dim \PP \, \Sym^\db (\CC^{n+1})^* + 1.
\end{split}
\end{equation}

\end{proof}

Proposition \ref{prop:DM27} implies that for generic $[\fb] \in \PP
\Symdb$, the entire fiber $p^{-1}([\fb])$ is contained in the locus
of $I$ where the differential $\diff p$ is smooth, hence the
corresponding Fano variety is smooth.

To prove the final part of Theorem \ref{thm:DM}, we must unfortunately
stray beyond the prerequisites mentioned at the beginning of these
notes; the required definitions and theorems can be found, for
example, in \cite{MR1658464}.
Assume that $\delm(n, \db, k) > 0$. Since $p_r$ is proper, there
exists a Stein factorization,
\begin{equation*}
p_r: I_r \stackrel{\rho}{\to} S \stackrel{\pi}{\to} \PP \, \Sym^\db (\CC^{n+1})^*,
\end{equation*}
with $S$ a normal variety, $\rho$ connected, and $\pi$ finite.
If $\pi$ is ramified. Since $\delta_- >0$, $p$, and hence $\pi$, is
surjective. If $\pi$ is ramified, by the Zariski-Nagata purity
theorem, $Z_k$ contains the inverse image of a divisor of $S$, and
hence $\codim(Z_k, I) \leq 1$, which contradicts the inequality (\ref{eq:codimZrIr}). 

\textbf{Acknowledgements:} These notes arose from collaboration with
Franz Kir\'aly, whom I would first like to thank. I would also like to
thank Frank Gounelas and Fabian M\"uller for several helpful
conversations, and Andreas H\"oring for pointing out a mistake in an
earlier version of the paper.